\documentclass[12pt]{amsart}
\usepackage{amssymb,mathrsfs,color,bbold}

\theoremstyle{plain}
\newtheorem{theorem}{Theorem}[section]
\newtheorem{lemma}[theorem]{Lemma}
\newtheorem{proposition}[theorem]{Proposition}

\theoremstyle{definition}

\newcommand{\term}[1]{{\textit{\textbf{#1}}}}

\newcommand{\abs}[1]{\lvert#1\rvert}
\newcommand{\norm}[1]{\lVert#1\rVert}
\newcommand{\bigabs}[1]{\bigl\lvert#1\bigr\rvert}
\newcommand{\bignorm}[1]{\bigl\lVert#1\bigr\rVert}

\title[Almost L-weakly compact operators]{Some results on almost L-weakly and almost M-weakly compact operators}

\date{\today}


\author[H. Li]{Hui Li}
\address{School of Mathematics, Southwest Jiaotong University, Chengdu, Sichuan, China, 610000.}
\email{lihuiqc@my.swjtu.edu.cn}

\author[Z. Chen]{Zili Chen}
\address{School of Mathematics, Southwest Jiaotong University, Chengdu, Sichuan, China, 610000.}
\email{zlchen@home.swjtu.edu.cn}

\begin{document}

\begin{abstract}
In this paper, we present some necessary and sufficient conditions for semi-compact operators being 
almost L-weakly compact (resp. almost M-weakly compact) and the converse. Mainly, we  prove that 
 if $X$ is a nonzero Banach space, then every semi-compact operator $T: X\rightarrow E$ is almost L-weakly 
 compact if and only if the norm of $E$ is order continuous. And every positive semi-compact 
 operator $T:E\rightarrow F$  is almost M-weakly compact 
if and only if the norm of $E'$ is order continuous. Moreover, we investigate the relationships
between almost L-weakly compact operators and Dunford-Pettis (resp. almost Dunford-Pettis) operators.
\end{abstract}

\maketitle

\section{introduction}

Throughout this paper, $X$ and $Y$ will denote real Banach spaces, $E$ and $F$ will 
denote real Banach lattices. $B_{X}$ (resp. $B_{E}$) is the closed unit of Banach 
space $X$ (resp. Banach lattice $E$) and $Sol(A)$ denotes the solid hull of a subset $A$ of a Banach lattice. 

Recall that a continuous operator $T:X\rightarrow E$ from a Banach space to a Banach lattice is 
said \term{semi-compact} if and only if for each $\varepsilon>0$ there exists some $u\in E_{+}$ such 
that $T(B_{X})\subset [-u, u]+\varepsilon B_{E}$.  In recent years, K. Bouras et al. \cite{BLM:18} introduced two classes 
of operators of almost L-weakly and almost M-weakly compact. Recall that an operator $T$ from a 
Banach space $X$ into a Banach lattice $F$ is called \term{almost L-weakly compact} if  $T$ 
carries relatively weakly compact subsets of $X$ onto L-weakly compact subsets of $F$. An 
operator $T$ from a Banach lattice $E$ into a Banach space $Y$ is called \term{almost M-weakly 
compact}  if for every disjoint sequence $(x_{n})$ in $B_{E}$ and every weakly convergent sequence $(f_{n})$ of $Y'$, we have $f_{n}(T(x_{n}))\rightarrow 0$. 

They proved in \cite{BLM:18} that an operator $T$ from a Banach space $X$ into a Banach
 lattice $F$  is almost L-weakly compact if and only if $f_{n}(T(x_{n}))\rightarrow 0$ for every 
 weakly convergent sequence $(x_{n})$ of $X$ and every disjoint sequence $(f_{n})$ 
 of $B_{F'}$ (\cite[Theorem 2.2]{BLM:18}). After that, A. Elbour et al. \cite{EAS:19} gave 
 a useful characterization of almost L-weakly compact operator. 
An operator  $T$ from a Banach space $X$ into a Banach lattice $F$  is almost L-weakly 
compact if and only if $T(X)\subset F^{a}$ and 
$f_{n}(T(x_{n}))\rightarrow 0$ for every weakly null sequence $(x_{n})$ of $X$ and 
every disjoint sequence $(f_{n})$ of $B_{F'}$ (\cite[Proposition 1]{EAS:19}). 

Recall that  a norm $\norm{\cdot}$ of a  Banach lattice $E$ is \term{order continuous} if
 for each net $(x_{\alpha})$ in $E$ with $x_{\alpha}\downarrow 0$, one has $\norm{x_{\alpha}}\downarrow 0$.  
It is easy to see that if $E$ has an order continuous norm, $E=E^{a}$. A  Banach lattice is said to have \term{weakly sequentially continuous lattice operations} whenever $x_{n}\xrightarrow{w} 0$ implies $\abs{x_{n}}\xrightarrow{w} 0$. 
Every $AM$-space has this property.  A  Banach space is said to have the \term{ Schur property} whenever  every  weakly null sequence 
is norm null, i.e., whenever $x_{n}\xrightarrow{w}0$ implies $\norm{x_{n}}\rightarrow 0$. A  Banach 
space is said to have the \term{ positive Schur property} whenever  every  disjoint weakly null 
sequence is norm null. In \cite{BLM:18}, it was proved  that the identity operator $Id_{E}$ is 
almost L-weakly compact if and only if $E$ has the positive Schur property (\cite[Proposition 2.2]{BLM:18}).
 And the identity operator $Id_{E}$ is almost M-weakly compact if and only if $E'$ has the positive
  Schur property (\cite[Corollary 2.1]{BLM:18}).

 Following from these conclusions, it is easy to see that there exist operators which are 
 semi-compact but not almost L-weakly compact or almost M-weakly compact.  And
  there also exist operators which are almost L-weakly compact (resp. almost M-weakly compact) but 
  not semi-compact.
 
  In this paper, we establish some necessary and sufficient conditions for semi-compact operator being 
  almost L-weakly compact (resp. almost M-weakly compact) and the converse. More precisely, 
  we  prove that every semi-compact operator $T$ from a nonzero Banach space $X$ to a Banach lattice $E$ is almost L-weakly compact if and only if the norm of $E$ is order continuous (Theorem \ref{sc-al}). 
   And every positive semi-compact operator from a  Banach lattice $E$ into a nonzero Banach lattice $F$ is 
   almost M-weakly compact if and only if the norm of $E'$ is order continuous (Theorem \ref{sc-am}). 
   We also investigate the conditions under which each almost L-weakly compact operator is 
   semi-compact (Theorems \ref{al-sc}, \ref{al-sc1}). Moreover, we show each positive almost
    Dunford-Pettis operator $T:E\rightarrow F$ is almost L-weakly compact  if and only if $F$ has 
    an order continuous norm (Proposition \ref{adp-alwc}).
  
 All operators in this paper are assumed to be continuous. We refer to ~\cite{AB:06,MN:91} 
 for all unexplained terminology and standard facts on vector and Banach lattices. All vector
  lattices in this paper are
assumed to be Archimedean.

\section{main results}

There exist operators which are semi-compact but not almost L-weakly compact. For example,
the identity operator $Id_{c}:c\rightarrow c$ is semi-compact since $c$ is an AM-space with unit.  
But it is not almost L-weakly compact since $c$ doesn't have the positive Schur property.

The following Theorem gives a necessary and sufficient condition under which every semi-compact
 operator is almost L-weakly compact.
\begin{theorem}\label{sc-al}
Let $X$ be a nonzero Banach space and $E$ be a Banach lattice.  Then the following statements are equivalent:
\begin{itemize}
\item[(1)] Every semi-compact operator $T:X\rightarrow E$ is almost L-weakly compact;
\item[(2)] The norm of $E$ is order continuous.
\end{itemize}
\end{theorem}
\begin{proof}
$(2)\Rightarrow (1)$ If the norm of $E$ is order continuous, then by Corollary 3.6.14 of \cite{MN:91}, semi-compact
 operator $T:X\rightarrow E$ is L-weakly compact. It is obvious that every L-weakly compact operator is almost 
 L-weakly compact. Hence $T$ is almost L-weakly compact.

$(1)\Rightarrow (2)$ Assume by way of contradiction that the norm of $E$ is not order continuous, we 
need to construct an operator which is semi-compact but not almost L-weakly compact.

Since the norm of $E$ is not order continuous, by Theorem 4.14 of \cite{AB:06}, there exists a 
vector $y\in E_{+}$ and a disjoint sequence $(y_{n})\subset [-y,y]$ such that $\norm{y_{n}}\nrightarrow 0$.  
On the other hand, as $X$ is nonzero,  we may fix $u\in X$ and pick a $\phi\in X'$ such 
that $\phi(u)=\norm{u}=1$ holds.

Now, we consider operator $T:X\rightarrow E$ defined by
$$T(x)=\phi(x)\cdot y$$
for each $x\in X$. Obviously, $T$ is semi-compact as it is compact (its rank is one). 
But it is not an almost L-weakly compact operator. If not, 
as the singleton $\{u\}$ is  a weakly compact subset of $X$, and $T(u)=\phi(u)\cdot y=y$, 
the singleton $\{y\}$  is an L-weakly compact subset of $E$.  Since disjoint 
sequence $(y_{n})\subset sol(\{y\})$, we have $\norm{y_{n}}\rightarrow 0$, which is a contradiction.
\end{proof}

There exist operators which are semi-compact  but not almost M-weakly compact. For example, 
the operator $T:\ell_{1}\rightarrow \ell_{\infty}$ defined by
$$T(\lambda_{n})=(\sum_{n=1}^{\infty}{\lambda_{n}})\cdot e$$
for each $(\lambda_{n}) \in \ell_{1}$, where $e=(1, 1,...)$ is the constant sequence with value 1 \cite[p. 322 ]{AB:06}. Obviously, 
$T$ is semi-compact as it is compact (its rank is one). But based on 
the argument in \cite[p. 3]{EAS:19}, we know that $T$ is not an almost M-weakly compact operator. 

The following Theorem gives a necessary and sufficient condition under which every
 semi-compact operator is almost M-weakly compact.

\begin{theorem}\label{sc-am}
Let $E$ and $F$ be two nonzero Banach lattices. Then the following 
statements are equivalent:
\begin{itemize}
\item[(1)] Every positive semi-compact operator $T:E\rightarrow F$ is almost M-weakly compact;
\item[(2)] The norm of $E'$ is order continuous.
\end{itemize}
\end{theorem}
\begin{proof}
$(2)\Rightarrow (1)$  Since positive operator $T:E\rightarrow F$ is semi-compact, following from 
Corollary 3.3 of \cite{AE:11}, $T': F'\rightarrow E'$ is an almost Dunford-Pettis operator. 
As the norm of $E'$ is order continuous, by Proposition 2.4 of \cite{BLM:18}, $T'$ is an 
almost L-weakly compact operator. Following from Theorem 2.5(1) of \cite{BLM:18}, $T$ 
is almost M-weakly compact.

$(1)\Rightarrow (2)$ Assume by way of contradiction that the norm of $E'$ is not order 
continuous, we need to construct a positive operator which is semi-compact but not almost M-weakly compact.

Since the norm of $E'$ is not order continuous, by Theorem 4.14 of \cite{AB:06}, there
 exists a vector $\phi\in E'_{+}$ and a disjoint sequence $(\phi_{n})\subset [-\phi,\phi]$ 
 such that $\norm{\phi_{n}}\nrightarrow 0$.  On the other hand, as $F$ is nonzero,  we 
 may fix $y \in F_{+}$ and pick a vector $g\in (F')_{+}$ such that $g(y)=\norm{y}=1$ holds.

Now, we consider operator $T:E\rightarrow F$ defined by
$$T(x)=\phi(x)\cdot y.$$
for each $x\in E$. Obviously, $T$ is positive and semi-compact operator as it is compact (its rank is one). 
But it is not an almost M-weakly compact operator. In fact, by 
Theorem 2.5(1) of \cite{BLM:18}, we only need to show that its adjoint $T: F'\rightarrow E'$ defined by 
$$T'(f)=f(y)\cdot \phi $$
for any $f\in F'$ is almost L-weakly compact. If not, 
as the singleton $\{g\}$ is  a weakly compact subset of $X'$,
 and $T'(g)=g(y)\cdot \phi=\phi$, the singleton $\{ \phi \}$ is an 
 L-weakly compact subset of $E'$.  Since disjoint sequence
  $(\phi_{n})\subset sol(\{ \phi \} )$, we have $\norm{\phi_{n}}\rightarrow 0$, which is a contradiction.
\end{proof}


There also exist operators which are almost L-weakly compact but not semi-compact. 
For instance, the identity operator $Id_{\ell_{1}}:\ell_{1}\rightarrow \ell_{1}$ is almost
 L-weakly compact since $\ell_{1}$ has the positive Schur property. But it is not 
 semi-compact. If not, as $\ell_{1}$ is discrete with order continuous norm, 
 $Id_{\ell_{1}}$ is compact, which is impossible. 

Next, denote $T:E\rightarrow F$ as a continuous operator ,  we investigate  
the conditions under which each almost L-weakly compact operator $T$ is  semi-compact. 

Based on Theorem 4 of \cite{EAS:19}, we know that if $E'$ has an order continuous norm, 
then each positive almost L-weakly compact operator $T$ is M-weakly compact, hence 
semi-compact. Now, we claim that if $E$ is reflexive then each almost L-weakly compact 
operator $T$ is semi-compact. In fact, if $E$ is reflexive, then $B_{E}$ is a relatively weakly 
compact subset of $E$. As $T$ is almost L-weakly compact, $T(B_{E})$ is an L-weakly compact subset of $F$. 
By Proposition 3.6.2 of \cite{MN:91}, for every $\varepsilon >0$, there exists 
a vector $u\in F_{+}^{a}\subset F_{+}$ such that $T(B_{E})\subset [-u, u]+\varepsilon B_{F}$. So $T$ is semi-compact.

The following Theorem gives the conditions under which each positive almost 
L-weakly compact operator $T$ from $E$ to $E$  is semi-compact.

\begin{theorem}\label{al-sc}
Let $E$ be a Banach lattice with an order continuous norm. Then the following assertions are equivalent:
\begin{itemize}
\item[(1)] Each positive almost L-weakly compact operator $T$ from $E$ to $E$ is semi-compact.
 \item[(2)] The norm of $E'$ is order continuous.
 \end{itemize}
\end{theorem}

\begin{proof}
$(2)\Rightarrow (1)$ Follows from Theorem 4 of \cite{EAS:19}.

$(1)\Rightarrow (2)$ Assume by way of contradiction that the norm of $E'$ is not order continuous.  
To finish the proof, we need to construct a positive almost L-weakly compact operator 
$T: E\rightarrow E$ which is not semi-compact.

Since the norm of $E'$ is not order continuous, it follows from Theorem 116.1
 of \cite{Zaanen:83} that there exists a norm bounded disjoint sequence $(u_{n})$
  of positive elements in $E$ which does not  weakly convergence to zero.  Without loss of generality, 
  we may assume that $\norm{u_{n}}\le 1$ for any $n$. And there exist $\varepsilon >0$ and $0\le \phi\in E'$ 
  such that $\phi(u_{n})>\varepsilon$ for all $n$.  Then by Theorem 116.3 of \cite{Zaanen:83}, 
  we know that the components $\phi_{n}$ of $\phi$ in the carriers $C_{u_{n}}$ form an order
   bounded disjoint sequence in $(E')_{+}$ such that 
 
\centerline{$\phi_{n}(u_{n})=\phi(u_{n})$\ for\ all\ $n$\ and \ $\phi_{n}(u_{m})=0$\ if \ $n\ne m.\qquad (*)$} 

Define the positive operator $S_{1}:E\rightarrow \ell_{1}$ as follows:
$$S_{1}(x)=\left(\frac{\phi_{n}(x)}{\phi(u_{n})}\right)_{n=1}^{\infty}$$
for all $x\in E$. Since 

$$\sum_{n=1}^{\infty}\bigabs{\frac{\phi_{n}(x)}{\phi(u_{n})}}\le \frac{1}{\varepsilon}\sum_{n=1}^{\infty}\phi_{n}(\abs{x})\le  \frac{1}{\varepsilon}\phi(\abs{x})$$
holds for all $x\in E$, the operator $S_{1}$ is well defined and it is also easy to 
see that $S_{1}$ is a positive operator.  

 Now define the operator $S_{2}: \ell_{1}\rightarrow E$ as follows:
$$S_{2}(\lambda_{n})=\sum_{{n=1}}^{\infty}\lambda_{n}u_{n}$$
for all $(\lambda_{n})\in \ell_{1}$.
As $\sum_{n=1}^{\infty}\norm{\lambda_{n}u_{n}}\le \sum_{n=1}^{\infty}\abs{\lambda_{n}}<\infty$, $S_{2}$ 
is well defined and is positive.

Next, we consider the composed operator $T=S_{2}\circ S_{1}:E\rightarrow \ell_{1}\rightarrow E$ defined by
$$T(x)=\sum_{n=1}^{\infty}\frac{\phi_{n}(x)}{\phi(u_{n})}u_{n}$$
for all $x\in E$. Now we claim $T$ is an almost L-weakly compact operator.  Since $E$ has an order 
continuous norm, $E=E^{a}$. It suffices to show that $T$ satisfies the condition (b) of Proposition 1 
of \cite{EAS:19}.  Let $x_{n}\xrightarrow{w} 0$ in $E$ and $(f_{n})$ be a disjoint sequence in $B_{E'}$. 
It is obvious that $S_{1}(x_{n})$ is  a weakly null sequence in $\ell_{1}$.  As $\ell_{1}$ has the Schur 
property, $\bignorm{S_{1}x_{n}}\rightarrow 0$. Hence $\norm{T(x_{n})}=\bignorm{S_{2}(S_{1}(x_{n}))}\rightarrow 0$. 
Now following from the inequality
$$\abs{f_{n}(T(x_{n}))}\le \norm{f_{n}}\bignorm{T(x_{n})}\le \bignorm{Tx_{n}},$$
we obtain that $T$ is an almost L-weakly compact operator.

But $T$ is not a semi-compact operator. In fact, note that $\norm{u_{n}}\le 1$ and $T(u_{n})=u_{n}$ for all $n$
 following from $(*)$.  So, if $T$ is semi-compact, then $T(B_{E})$ is almost order bounded in $E$. Hence, 
 $(u_{n})\subset T(B_{E})$ is also almost order bounded. So, $u_{n}\rightarrow 0$ in $\sigma(E,E')$, 
 which is a contradiction.  
\end{proof}

To investigate the necessary and sufficient conditions under which each almost L-weakly  
compact operator $T:E\rightarrow F$ is semi-compact, we first give the following useful Lemma.

\begin{lemma}\label{dn}
Let $E$ be a Banach lattice with an order continuous norm. If $(u_{n})$ is a norm bounded disjoint 
sequence of $E$ such that the set $\{u_{n}\}$ is almost order bounded in $E$, then $(u_{n})$ 
converges to zero in norm.
\end{lemma}

\begin{proof}
Since $A: =\{u_{n}: n\in \mathcal{N}\}$ is almost order bounded, there exists some 
$u\in E_{+}$ such that 

\centerline{$\bignorm{(\abs{u_{n}}-u)^{+}}\le \varepsilon$\ for\ all \ $n$.$$}

 On the other hand, since $(\abs{u_{n}}\wedge u)$ is an order bounded disjoint sequence
  in $E$ and the norm of $E$ is order continuous, following from Theorem 4.14 of \cite {AB:06}, 
  $(\abs{u_{n}}\wedge u)$ converges to zero in norm . Hence, there exists some $n_{0}$ such that 
 
\centerline{$\bignorm{\abs{u_{n}}\wedge u}\le \varepsilon$  \ for \ all \ $n\ge n_{0}$.}

Now, following from the equality $\abs{u_{n}}=(\abs{u_{n}}-u)^{+}+\abs{u_{n}}\wedge u$, we obtain that 

$$\norm{u_{n}}\le \bignorm{(\abs{u_{n}}-u)^{+}}+\bignorm{\abs{u_{n}}\wedge u} \le 2\varepsilon$$
holds for all $n\ge n_{0}$. So, $u_{n}\rightarrow 0$ in norm.
\end{proof}

\begin{theorem}\label{al-sc1}
Let $E$ and $F$ be two Banach lattices such that the norm of $F$ is order continuous. 
Then the following assertions are equivalent:
\begin{itemize}
\item[(1)] Each positive almost L-weakly compact operator $T:E\rightarrow F$ is semi-compact;

\item[(2)]One of the following conditions is valid:
\begin{itemize}
\item[(a)] The norm of $E'$ is order continuous;
\item[(b)] $E$ or $F$ is finite dimensional.
\end{itemize}
\end{itemize}
\end{theorem}

\begin{proof}
$(1)\Rightarrow (2)$
Assume $E$ and $F$ are both infinite dimensional. We have to show that  the norm 
of $E'$ is  order continuous.  If not, to finish the proof, we need to construct a positive almost
 L-weakly compact operator $T: E\rightarrow F$ which is not semi-compact.

Since the norm of $E'$ is not order continuous, similarly with the proof of Theorem \ref{al-sc}, 
we define the positive operator $S_{1}:E\rightarrow \ell_{1}$ as follows:
$$S_{1}(x)=\left(\frac{\phi_{n}(x)}{\phi(u_{n})}\right)_{n=1}^{\infty}$$
for all $x\in E$. And  the operator $S_{1}$ is well defined. 

On the other hand, since $F$ is infinite dimensional, by Lemma 2.3 of \cite{AEH:11}, there exists a 
disjoint sequence $(y_{n})\subset (B_{F})_{+}$  such that $\norm{y_{n}}=1$. Now define the operator $S_{3}: \ell_{1}\rightarrow F$ as follows:
$$S_{3}(\lambda_{n})=\sum_{{n=1}}^{\infty}\lambda_{n}y_{n}$$
for all $(\lambda_{n})\in \ell_{1}$.
As $\sum_{n=1}^{\infty}\norm{\lambda_{n}y_{n}}\le \sum_{n=1}^{\infty}\abs{\lambda_{n}}<\infty$, $S_{3}$ is 
well defined and is positive.

Next, we consider the composed operator $T=S_{3}\circ S_{1}:E\rightarrow \ell_{1}\rightarrow F$ defined by
$$T(x)=\sum_{n=1}^{\infty}\frac{\phi_{n}(x)}{\phi(u_{n})}y_{n}$$
for all $x\in E$. Now we claim $T$ is an almost L-weakly compact operator.  Since $F$ has an order 
continuous norm, $F=F^{a}$. It suffices to show that $T$ satisfies the condition (b) of 
Proposition 1 of \cite{EAS:19}.  Let $x_{n} \xrightarrow{w} 0$ in $E$ and $(f_{n})$ be a 
disjoint sequence in $B_{F'}$. It is obvious that $S_{1}(x_{n})$ is  a weakly null sequence 
in $\ell_{1}$.  As $\ell_{1}$ has the Schur property, $\bignorm{S_{1}x_{n}}\rightarrow 0$. 
Hence $\norm{T(x_{n})}=\bignorm{S_{3}(S_{1}(x_{n}))}\rightarrow 0$. Now following from 
the inequality
$$\abs{f_{n}(T(x_{n}))}\le \norm{f_{n}}\bignorm{T(x_{n})}\le \bignorm{Tx_{n}},$$
we obtain that $T$ is an almost L-weakly compact operator.

But $T$ is not a semi-compact operator. In fact, note that $\norm{u_{n}}\le1$ and $T(u_{n})=y_{n}$ 
for all $n$ following from $(*)$.  So, if $T$ is semi-compact, then $T(B_{E})$ is almost order bounded in $F$. 
Hence, $(y_{n})\subset T(B_{E})$ is almost order bounded. By Lemma \ref{dn}, $y_{n}\rightarrow 0$ 
in norm, which is a contradiction.  
 
$(2a)\Rightarrow (1)$ Follows from Theorem 4 of \cite{EAS:19}.

$(2b)\Rightarrow (1)$ Let $T: E\rightarrow F$ be a positive operator. If $E$ is finite dimensional, $T$ is M-weakly compact. 
Also, if $F$ is finite dimensional, $T$ is L-weakly compact.  Hence, $T$ is semi-compact.
\end{proof}

There exist operators which are almost M-weakly compact but not semi-compact. For instance, the 
identity operator $Id_{c_{0}}:c_{0}\rightarrow c_{0}$ is almost M-weakly compact since
 $(c_{0})'=\ell_{1}$ has the Positive Schur property. But it is not semi-compact. If not, as $c_{0}$ is
  discrete with order continuous norm, $Id_{c_{0}}$ is compact, which is impossible. Following by 
  Corollary 5 of \cite{EAS:19}, we have the following assertion.

\begin{theorem}
Let $E$ and $F$ be two Banach lattices. And let $T:E\rightarrow F$ be an order bounded almost 
M-weakly compact operator. If $F''$ has order continuous norm, then $T$ is  semi-compact.
\end{theorem}

By Theorem 1 of \cite{EAS:19}, we know that every Dunford-Pettis operator $T: X\rightarrow E$ is almost 
L-weakly compact if and only if $E$ has an order continuous norm.  Next, we investigate the conditions 
under which each almost L-weakly compact operator is Dunford-Pettis. 

There exist operators which are almost L-weakly compact but not Dunford-Pettis. For example, the identity 
operator $Id_{L_{1}[0,1]}:L_{1}[0,1]\rightarrow L_{1}[0,1]$ is almost L-weakly compact as $L_{1}[0,1]$ has 
the positive Schur property. But it is not Dunford-Pettis as $L_{1}[0,1]$ does not have the Schur 
property. However, we have the following result.

\begin{proposition}
Let $X$ be a Banach space and $E$ be a Banach lattice. If $E$ has weakly sequentially continuous lattice operations, then each almost L-weakly compact operator $T:X\rightarrow E$ is Dunford-Pettis. 
\end{proposition}

\begin{proof}
Let $(x_{n})$ be a weakly null sequence, we have to show that $T(x_{n})\rightarrow 0$ in norm. 
Based on  Corollary 2.6 of \cite{DF:79}, it suffices to show $\abs{T(x_{n})}\xrightarrow{w} 0$ in 
$E$ and $f_{n}(T(x_{n}))\rightarrow 0$ for each positive disjoint sequence $(f_{n})$ in $B_{E'}$. 

As $x_{n}\xrightarrow{w} 0$, then $T(x_{n})\xrightarrow{w} 0$. Since the lattice operations of $E$ are 
weakly sequentially continuous,  $\abs{T(x_{n})}\xrightarrow{w} 0$ in $E$. On the other hand, as $x_{n}\xrightarrow{w} 0$ in $X$ and $T$ is an 
almost L-weakly compact operator, for each positive disjoint sequence $(f_{n})$ in
 $B_{E'}$, $f_{n}(T(x_{n}))\rightarrow 0$.  Hence, we get that $T$ is Dunford-Pettis.
\end{proof}

At last, we give a conclusion about the relationships between almost L-weakly compact operators and 
almost Dunford-Pettis operators.
K. Bouras et al. show that if $F$ has an order continuous norm,  each positive almost Dunford-Pettis
 operator $T:E\rightarrow F$ is almost L-weakly compact (\cite[Proposition 2.4]{BLM:18}).  We show 
 that the condition of ``$F$ has an order continuous norm'' is also necessary.

\begin{proposition}\label{adp-alwc}
Let $E$ and $F$ be two nonzero Banach lattices. Then the following statements are equivalent:
\begin{itemize}
\item[(1)] Each positive almost Dunford-Pettis operator $T:E\rightarrow F$ is almost L-weakly compact.
\item[(2)] The norm of $F$ is order continuous.
\end{itemize}
\end{proposition}
\begin{proof}
$(2)\Rightarrow (1)$ Follows from Proposition 2.4 of \cite{BLM:18}.

$(1)\Rightarrow (2)$ Assume by way of contradiction that the norm of $F$ is not order continuous. 
we need to construct a positive operator which is almost Dunford-Pettis but not almost L-weakly compact.

Similarly with the proof of Theorem \ref{sc-al}. Since the norm of $F$ is not order continuous, by
 Theorem 4.14 of \cite{AB:06}, there exists a vector $y\in F_{+}$ and a disjoint sequence $(y_{n})\subset [-y,y]$ 
 such that $\norm{y_{n}}\nrightarrow 0$.  On the other hand, as $E$ is nonzero,  we may fix $u\in E_{+}$ and
  pick a $\phi\in (E')_{+}$ such that $\phi(u)=\norm{u}=1$ holds.

Now, we consider operator $T:E\rightarrow F$ defined by
$$T(x)=\phi(x)\cdot y$$
for each $x\in E$. Obviously, $T$ is a positive operator and is compact (its rank is one).  Hence, it is almost Dunford-Pettis. 
But it is not an almost L-weakly compact operator. If not, as the 
 singleton $\{u\}$ is  a weakly compact subset of $E$, and $T(u)=\phi(u)\cdot y=y$, the
  singleton $\{y\}$ is an L-weakly compact subset of $F$.  Since disjoint sequence 
  $(y_{n})\subset sol(\{y\})$, we have $\norm{y_{n}}\rightarrow 0$, which is a contradiction.

\end{proof}


\end{document}